\documentclass[11pt]{amsart}

\usepackage[mathscr]{eucal}
\usepackage{amsmath,amssymb,amsfonts,amsthm,enumerate}

\newtheorem{theorem}{Theorem}[section]
\newtheorem{lemma}[theorem]{Lemma}

\newtheorem{theirtheorem}{Theorem}

\newcommand{\Z}{\mathbb Z}
\newcommand{\R}{\mathbb R}

\DeclareMathOperator{\supp}{Supp}

\newcommand{\la}{\langle}
\newcommand{\ra}{\rangle}
\newcommand{\be}{\begin{equation}}
\newcommand{\ee}{\end{equation}}
\newcommand{\und}{\;\mbox{ and }\;}
\newcommand{\nn}{\nonumber}
\newcommand{\ber}{\begin{eqnarray}}
\newcommand{\eer}{\end{eqnarray}}
\newcommand{\Sum}[2]{\underset{#1}{\overset{#2}{\sum}}}
\newcommand{\Summ}[1]{\underset{#1}{\sum}}

\newcommand{\A}{\mathscr A}

\newcommand{\sP}{\mathscr S}
\newcommand{\Fc}{\mathcal F}
\newcommand{\vp}{\mathsf v}
\newcommand{\h}{\mathsf h}

\newcommand{\sA}{\mathscr{A}}

\DeclareSymbolFont{goo}{OMS}{cmsy}{b}{n}
\DeclareMathSymbol{\gooT}{\mathalpha}{goo}{"1}
\newcommand{\bdot}{\mathbin{\gooT}}
\begin{document}

\title{Iterated Sumsets and  Setpartitions}
\author{David J. Grynkiewicz}
\email{diambri@hotmail.com}
\address{Department of Mathematical Sciences, University of Memphis, Memphis,  TN 38152, USA}
\subjclass[2010]{11B75}
\keywords{zero-sum, sumset, subsequence sum, subsum, Partition Theorem, DeVos-Goddyn-Mohar Theorem, Kneser's Theorem}

\begin{abstract}
Let $G\cong \Z/m_1\Z\times\ldots\times \Z/m_r\Z$ be a finite abelian group with $m_1\mid\ldots\mid m_r=\exp(G)$. The $n$-term subsums version of Kneser's Theorem,  obtained either via the DeVos-Goddyn-Mohar Theorem or the Partition Theorem, has
become a powerful tool used to prove numerous zero-sum and subsequence sum questions. It provides a structural description of sequences having a small number of $n$-term subsequence sums, ensuring this is only possible if most terms of the sequence are contained in a small number of $H$-cosets. For large $n\geq \frac1p|G|-1$ or $n\geq \frac1p|G|+p-3$, where $p$ is the smallest prime divisor of $|G|$, the structural description is particularly strong.  In particular, most terms of the sequence become contained in a single $H$-coset, with additional properties holding regarding the representation of elements of $G$ as subsequence sums. This strengthened form of the subsums version of Kneser's Theorem  was later to shown to hold under the weaker hypothesis $n\geq \mathsf d^*(G)$, where $\mathsf d^*(G)=\Sum{i=1}{r}(m_i-1)$. In this paper, we reduce the restriction on $n$ even further to an optimal, best-possible value, showing we need only assume $n\geq \exp(G)+1$ to obtain the same conclusions, with the bound further  improved for several classes of near-cyclic groups.\end{abstract}

\maketitle

\section{Notation and Overview}

Let $G$ be an abelian group and let $A,\,B\subseteq G$ be finite and nonempty subsets. Their sumset is defined as $A+B=\{a+b:\;a\in A,\,b\in B\}$. For $x\in G$, we let $\mathsf r_{A+B}(x)=|(x-B)\cap A|=|(x-A)\cap B)|$ denote the number of ways to represent $x=a+b$ as an element in the sumset $A+B$, where $(a,b)\in A\times B$. When $\mathsf r_{A+B}(x)=1$, we say that $x$ is a \emph{unique expression} element in $A+B$. Note $A+B=\{x\in G:\;\mathsf r_{A+B}(x)\geq 1\}$. Multiple summand sumsets are defined analogously: $\Sum{i=1}{n}A_i=\{\Sum{i=1}{n}a_i:\;a_i\in A_i\}$ for subsets $A_1,\ldots,A_n\subseteq G$. For an integer $n\geq 0$, we use the abbreviation $nA={\underbrace{A+\ldots+A}}_{n}$, where $0A:=\{0\}$, for the $n$-fold iterated sumset.

The \emph{stabilizer} of $A\subseteq G$ is the subgroup $\mathsf H(A)=\{x\in G:\; x+A=A\}\leq G$. It is the maximal subgroup $H$ such that $A$ is a union of $H$-cosets. When $\mathsf H(A)$ is trivial,  $A$ is called \emph{aperiodic}, and when $\mathsf H(A)$ is nontrivial,  $A$ is called \emph{periodic}. More generally, if $A$ is a union of $H$-cosets for some subgroup $H\leq G$ (necessarily with $H\leq \mathsf H(A)$), then $A$ is called \emph{$H$-periodic}.

If $H\leq G$ is a subgroup, then we let $\phi_H:G\rightarrow G/H$ denote the natural homomorphism. Note, if $H=\mathsf H(A)$, then $\phi_H(A)$ is aperiodic. We use $H<G$ to indicate that $H$ is proper, and $$\la A\ra_*:=\la A-A\ra=\la -x+A\ra\quad\mbox{ for any $x\in A$}$$ denotes the subgroup generated affinely by $A$, which is the smallest subgroup $H$ such that $A$ is contained in an $H$-coset.

Regarding sequences and subsequence sums, we follow the standardized notation from Factorization Theory  \cite{gao-ger-survey} \cite{alfredbook} \cite{Gbook}.  The key parts are summarized  here. Let $G_0\subseteq G$ be a subset. A \emph{sequence} $S$ of terms from $G_0$ is viewed formally as an element of the free abelian monoid  with basis $G_0$, denoted  $\mathcal F(G_0)$. Thus a sequence $S\in \Fc(G_0)$ is written as a finite multiplicative string of terms, using the bold dot operation $\bdot$ to concatenate terms, and with the order irrelevant:
$$S=g_1\bdot\ldots\bdot g_\ell$$ with $g_i\in G_0$ the terms of  $S$ and $|S|:=\ell\geq 0$ the \emph{length} of $S$.
 Given $g\in G_0$ and $s\geq 0$,  we let $g^{[s]}={\underbrace{g\bdot\ldots\bdot g}}_s$ denote the sequence consisting of the element $g$ repeated $s$ times.  We let
$$\vp_g(S)=|\{i\in [1,\ell]:\;g_i=g\}|\geq 0$$
denote the multiplicity of the term $g\in G_0$ in the sequence $S$.
If $S,\,T\in \Fc(G_0)$ are sequences, then  $S\bdot T\in \Fc(G_0)$ is  the sequence obtained by concatenating the terms of $T$ after those of $S$.
  A sequence $S$ may also be defined by listing its terms as a product: $S=\prod^\bullet_{g\in G_0}g^{[\vp_g(S)]}.$
 We use $T\mid S$ to indicate that $T$ is a subsequence of $S$ and let ${T}^{[-1]}\bdot S$ or $S\bdot {T}^{[-1]}$ denote the sequence obtained by  removing the terms of $T$ from $S$.
Then
\begin{align*}
 &\h (S) =     \max \{ \mathsf v_g(S):\; g \in G_0 \} \quad  \mbox{is the \emph{maximum  multiplicity} of $S$},\\
 &\supp(S)=\{g\in G_0:\; \vp_g(S)>0\}\subseteq G \quad\mbox{ is the \emph{support} of $S$},\\
 &\sigma(S)=\Sum{i=1}{\ell}g_i=\Summ{g\in G_0}\vp_g(S)g\in G\quad\mbox{ is the \emph{sum} of $S$},\\
 &\Sigma_n(S)=\{\sigma(T):\; T\mid S, \, |T|=n\}\subseteq G \quad \mbox{ are the $n$-term \emph{sub(sequence)-sums} of $S$},\\
 &\Sigma(S)=\{\sigma(T):\;T\mid S,\,|T|\geq 1\}\subseteq G\quad\mbox{ are the \emph{sub(sequence)-sums} of $S$}.
\end{align*}
Given a map
$\varphi \colon G_0 \to G'_0$, we let $\varphi(S)=\varphi(g_1)\bdot\ldots\bdot \varphi(g_\ell)\in \Fc(G'_0)$. The sequence $S$ is called \emph{zero-sum} if $\sigma(S)=0$. A \emph{setpartition} $\sA=A_1\bdot\ldots\bdot A_n$ over $G_0$ is a sequence of \emph{finite}, \emph{nonempty} subsets $A_i\subseteq G_0$. A setpartition naturally partitions its underlying sequence $$\mathsf{S}(\mathscr A):={\prod}^\bullet_{i\in [1,n]}{\prod}^\bullet_{g\in A_i}g\in \Fc(G_0)$$ into $n$ sets, so $\mathsf S(\sA)$ is the sequence obtained by concatenating the elements from every  $A_i$. We let $\sP(G_0)$ denote the set of all setpartitions over $G_0$, and refer to a setpartition of length $|\A|=n$ as an \emph{$n$-setpartition}.

Intervals are discrete, so $[a,b]=\{x\in \Z:\; a\leq x\leq b\}$ for $a,\,b\in \R$, as are variables introduced with inequalities. For $m\geq 1$, we let $C_m\cong \Z/m\Z$ denote a cylic group of order $m$. If $G$ is finite, then $G\cong C_{m_1}\times\ldots\times C_{m_r}$ for some $m_1\mid \ldots\mid m_r$ with $m_r=\exp(G)$ the \emph{exponent} of $G$. The \emph{Davenport Constant}, denoted $\mathsf D(G)$, is the least integer such  that a sequence of terms from $G$ with length $|S|\geq \mathsf D(G)$ must always contain a nontrivial zero-sum subsequence. In general, $\mathsf d^*(G)+1\leq \mathsf D(G)\leq |G|$, where $\mathsf d^*(G):=\Sum{i=1}{r}(m_i-1)$, though both inequalities may fail 
(see \cite[Propositions 5.1.4 and 5.1.8, pp. 341]{alfredbook}, or \cite{wolfgang-ordaz-olson-constant} for related results regarding the strong Davenport constant).

\bigskip

Subsequence sums and zero-sums  have been studied as an independent topic in Combinatorial Number Theory for many years and are now an important tool for those interested in Factorization Theory over Krull Domains and other monoids (see  \cite{gao-ger-survey} \cite{Alfred-Ruzsa-book} \cite{alfredbook} \cite{Gbook}). Their study often utilizes results from  Inverse Additive Number Theory, which seeks to characterize the structure of small cardinality sumsets. One of the key starting points here  is Kneser's classical theorem for  sumsets \cite[Theorem 4.1.1]{Alfred-Ruzsa-book} \cite[Theorem 5.2.6]{alfredbook} \cite[Theorem 6.1]{Gbook} \cite{kneserstheorem} \cite[Theorem 4.1]{natboook} \cite[Theorem 5.5]{taobook}.

\begin{theirtheorem}[Kneser's Theorem] Let $G$ be an abelian group and let $A_1,\ldots,A_n\subseteq G$ be finite, nonempty subsets. Then $$|\Sum{i=1}{n}A_i|\geq \Sum{i=1}{n}|A_i+H|-(n-1)|H|=\Sum{i=1}{n}|A_i|-(n-1)|H|+\rho,$$ where $H=\mathsf H(\Sum{i=1}{n}A_i)$ and $\rho:=\Sum{i=1}{n}|(A_i+H)\setminus A_i|$.
\end{theirtheorem}

Note $\Sum{i=1}{n}A_i=\Sum{i=1}{n}(A_i+H)$, and $\rho$ measures the number of ``holes'' in the sets $A_i$ relative to the sets $A_i+H$. Kneser's Theorem first appeared in the 1960s \cite{kneserstheorem}. It took much longer for the analogous result for $n$-term subsums to be developed, which is proved either as a special case of the DeVos-Goddyn-Mohar Theorem  or the Partition Theorem (see the discussion in \cite[pp. 181--182]{Gbook}).

\begin{theirtheorem}[Subsum Kneser's Theorem]\label{thm-subsum-kneser}
Let $G$ be an abelian group, let $n\geq 1$, let $S\in \Fc(G)$ be a sequence  with $\mathsf h(S)\leq n\leq |S|$, let $H=\mathsf H(\Sigma_n(S))$,   let $X\subseteq G/H$ be the subset of all $x\in G/H$ for which $x$ has multiplicity at least $n$ in $\phi_H(S)$, and let $e$ be the number of terms from $S$ not contained in $\phi_H^{-1}(X)$. Then \ber\nn |\Sigma_n(S)|&\geq& (|S|-n+1)-(n-e-1)(|H|-1)+\rho,\\ \nn&=& |S|-(n-1)|H|+e(|H|-1)+\rho,\eer where $\rho=|X||H|n+e-|S|\geq 0$.
\end{theirtheorem}

The bound given in Theorem \ref{thm-subsum-kneser} is equal to $$((N-1)n+e+1)|H|=(\Summ{x\in G/H}\min\{n,\,\vp_x(\phi_H(S))\}-n+1)|H|,$$ where $N=|X|$, which is how the bound is stated in \cite{Gbook} and \cite{DGM}. The form given above is often more practical and highlights the connection with Kneser's Theorem better.  If we define $S^*$ to be the sequence obtained from $S$ (as given in Theorem \ref{thm-subsum-kneser}) by taking each term $x\in \phi_H^{-1}(X)$ and changing its multiplicity from $\vp_x(S)$ to $\vp_x(S^*)=n$, then $S\mid S^*$, $|S^*|=|S|+\rho$ and $\Sigma_n(S)=\Sigma_n(S^*)$ with $\rho$ measuring the number of ``holes'' in the sequence $S$ relative to $S^*$. The sequence $S^*$ plays the same role in Theorem \ref{thm-subsum-kneser} as the sets $A_i+H$ in the bound $|\Sum{i=1}{n}A_i|\geq \Sum{i=1}{n}|A_i+H|-(n-1)|H|$ obtained from Kneser's Theorem.
Theorem \ref{thm-subsum-kneser} can be obtained either from the DeVos-Goddyn-Mohar Theorem or the  Partition Theorem. The Partition Theorem first appeared (in some form) in \cite{ccd}, with the variation allowing $S'\mid S$ appearing in \cite{hamconj}. The more general form given below, which subtlety refines and strengthens the Subsum Kneser's Theorem, may be found in \cite[Theorem 14.1]{Gbook}, slightly reworded here.

\begin{theirtheorem}[Partition Theorem]\label{thm-partition-thm} Let $G$ be an abelian group, let $n\geq 1$, let $S\in \Fc(G)$ be a sequence, let $S'\mid S$ be a subsequence with $\mathsf h(S')\leq n\leq |S'|$, let $H=\mathsf H(\Sigma_n(S))$, let $X\subseteq G/H$ be the subset of all $x\in G/H$ for which $x$ has multiplicity at least $n$ in $\phi_H(S)$, and let $e$ be the number of terms from $S$ not contained in $\phi_H^{-1}(X)$.  Then there exists a setpartition $\sA=A_1\bdot\ldots\bdot A_n\in \sP(G)$ with $\mathsf S(\sA)\mid S$ and $|\mathsf S(\sA)|=|S'|$ such that
either
\begin{itemize}
\item[1.] $|\Sigma_n(S)|\geq |\Sum{i=1}{n}A_i|\geq \Sum{i=1}{n}|A_i|-n+1=|S'|-n+1$, or
\item[2.] $|\Sigma_n(S)|=|\Sum{i=1}{n}A_i|\geq \Sum{i=1}{n}|A_i+H|-(n-1)|H|=|S'|-(n-1)|H|+e(|H|-1)+\rho$, where  $\rho=|X||H|n+e-|S'|\geq 0$, while  $\supp(\mathsf S(\sA)^{[-1]}\bdot S)\subseteq \phi_H^{-1}(X)\subseteq A_i+H$ and $|A_i\setminus \phi_H^{-1}(X)|\leq 1$ for all $i\in [1,n]$.
\end{itemize}
\end{theirtheorem}

The Partition Theorem and Subsum Kneser's Theorem are important tools that have been used as a key ingredient in the proof of many results about subsums and zero-sums, including the results from \cite{arie-II} \cite{PropB} \cite{Gao-nsum-paper} \cite{Gao-bonus} \cite{GG-nonabelian-index2} \cite{nondescrea-diam-zerosum} \cite{hamconj} \cite{hypergraph-egz} \cite{wegz} \cite{number-zs} \cite{GrahamConj} \cite{rasheed} \cite{andy-paper} \cite{oscar1} \cite{oscar-gaothms}. These results often involve  generalizing prior results,  confirming or partially confirming conjectures, or deal with  questions partially tackled by other authors, including results from
\cite{bialostocki-I} \cite{Bialostocki-II} \cite{bialostock-nondecreasing} \cite{bollobas-leader} \cite{Cao} \cite{Caro-WEGZ-survey} \cite{graham-original} \cite{grahamconj} \cite{Furedi-kleitman-m-term-zs} \cite{Gao-preolson} \cite{Gao-conj-ordaz-paper} \cite{ham-ordaz} \cite{ham-subsumConj} \cite{kisin} \cite{Mann-preolson} \cite{olson1-pregao} \cite{pingzhui-zeng}. While the form given in Theorem \ref{thm-subsum-kneser} is in some cases  sufficient, the added refinements given in Theorem \ref{thm-partition-thm} can be helpful, particularly when it can be assumed that $|X|=1$. Indeed, early forms of the Partition Theorem were formulated to focus particularly on  sequences $S$ and values of $n$ for which $|X|=1$ could be guaranteed, in which case either $|\Sigma_n(S)|$ could be guaranteed to be large or most terms of $S$ shown to lie in a common $K$-coset for some subgroup $K\leq G$. In the latter case, letting $S_K\mid S$ denote the subsequence of terms contained in this $K$-coset, it is additionally useful to know that $\Sigma_k(S_K)$ achieves its maximal possible size, when it is an entire $K$-coset, ideally for small $k$ and with $\Sigma_k(S_K)=\Sigma_k(S'_K)$ for a small length subsequence $S'_K\mid S_K$, as this frees up the remaining terms of $S$ to be used for other means. Various such versions have been given for $n\geq \frac1p|G|-1$ or $n\geq \frac1p|G|+p-3$ or $n\geq \mathsf d^*(G)$, where $p$ is the smallest prime divisor of $|G|$. See  \cite{oscar-weighted} \cite{ccd} \cite{hamconj} \cite[Theorems 15.1 and 15.2]{Gbook}. For instance, one of the main results of \cite{oscar-weighted} is a weighted version of the following theorem.

\begin{theirtheorem}\label{thm-oscar-partition-cor}Let $G$ be a finite abelian group, let $n\geq 1$, let $S\in \Fc(G)$ be a sequence of terms from $G$, and let   $S'\mid S$ be a subsequence with $\mathsf h(S')\leq n\leq |S'|$. Suppose $n\geq \mathsf d^*(G)$. Then there is a setpartition $\mathscr A=A_1\bdot\ldots\bdot A_n\in \sP(G)$ with $\mathsf S(\mathscr A)\mid S$ and $|\mathsf S(\mathscr A)|=|S'|$ such that
 \begin{itemize}
 \item[(i)] $|\Sigma_n(S)|\geq |\Sum{i=1}{n}A_i|\geq \min\{|G|,\, |S'|-n+1\}$, or
 \item[(ii)] there exists a proper, nontrivial subgroup $K<G$ and $\alpha\in G$ such that the following hold:
\begin{itemize}
\item[(a)] $\supp(\mathsf S(\mathscr A)^{[-1]}\bdot S)\subseteq \alpha+K$,
\item[(b)]  $|\Sigma_n(S)|\geq (e_K+1)|H|$, where $e_K\leq |G/K|-2$ is the number of terms of $S$ lying outside the coset $\alpha+K$,
\item[(c)]  $(\alpha+K)\cap A_i\neq \emptyset$ for all $i\in [1,n]$, and $A_i\subseteq \alpha+K$ for all $i\in [1,\mathsf d^*(K)]$,
    \item[(d)] $\Sum{i=1}{\mathsf d^*(K)} A_i=\mathsf d^*(G)\alpha+K$.
\end{itemize}
  \end{itemize}
\end{theirtheorem}

Theorem \ref{thm-oscar-partition-cor} ensures that all but $|G/K|-2$ terms of $S$ are from the same coset $\alpha+K$ and that an entire $K$-coset can be represented using $k$-term subsequence sums with $k\leq \mathsf d^*(K)$. This improved a similar result valid for $n\geq \frac{1}{p}|G|-1$  \cite[Theorem 3]{hamconj}, where $p$ is the smallest prime divisor of $|G|$, which was the form of Partition Theorem most often employed in its earliest applications, and gave only slightly weaker conclusions than a variation implicitly known to be valid for $n\geq \frac{1}{p}|G|+p-3$ (see \cite[Exercise 15.2]{Gbook}). The goal of  this paper is to provide a single generalization of all such strengthened versions of the Partition Theorem  valid for a much more lenient, optimal value of $n$.

\begin{theorem}
\label{thm-partitoincor}
Let $G$ be a finite abelian group, let $n\geq 1$, let $S\in \Fc(G)$ be a sequence of terms from $G$ with $H=\mathsf H(\Sigma_n(S))$, and let   $S'\mid S$ be a subsequence with $\mathsf h(S')\leq n\leq |S'|$. Suppose either $H$ is trivial, equal to $G$, or that one of the following holds:
\begin{itemize}
\item[1.] $n\geq \exp(G/H)+1$,
\item[2.] $n\geq \exp(G/H)>|H|$,
\item[3.] $n\geq \exp(G/H)$ and $G/H\cong C_2\times C_{\exp(G/H)}$,
\item[4.] $n\geq \exp(G/H)-1$ and $G/H$ is cyclic.
\end{itemize}
Then there is a setpartition $\mathscr A=A_1\bdot\ldots\bdot A_n\in \sP(G)$ with $\mathsf S(\mathscr A)\mid S$ and $|\mathsf S(\mathscr A)|=|S'|$ such that
 \begin{itemize}
 \item[(i)] $|\Sigma_n(S)|\geq |\Sum{i=1}{n}A_i|\geq \min\{|G|,\, |S'|-n+1\}$, or
 \item[(ii)] there exists a nontrivial subgroup $K\leq H<G$ and $\alpha\in G$ such that the following hold:
\begin{itemize}
\item[(a)] $\Sigma_n(S)=\Sum{i=1}{n}A_i$ and $\supp(\mathsf S(\mathscr A)^{[-1]}\bdot S)\subseteq \alpha+K$,
\item[(b)]  $|\Sigma_n(S)|\geq (e_H+1)|H|$, where $e_H\leq \min\{|G/H|-2,\, \frac{|S'|-n}{|H|}-1\}$ is the number of terms of $S$ lying outside the coset $\alpha+H$, and  $|\Sigma_n(S)|\geq (e_K+1)|K|$, where $e_K\leq \min\{|G/K|-2,\, \frac{|S'|-n}{|K|}-1\}$ is the number of terms of $S$ lying outside  $\alpha+K$.
\item[(c)]  $(\alpha+K)\cap A_i\neq \emptyset$ for all $i\in [1,n]$, \ $A_i\subseteq \alpha+K$ for all $i\in [1,n-e_K]$, and $|A_i\setminus (\alpha+K)|=1$ for all $i\in [n-e_k+1,n]$,
    \item[(d)] $\Sum{i=1}{n-e_K} A_i=(n-e_K)\alpha+K$.
\end{itemize}
  \end{itemize}
\end{theorem}

The proof of Theorem \ref{thm-partitoincor} will make use of a recent result characterizing the structure of an $n$-fold sumset with small sumset $|nA|<n|A|$ (Theorem \ref{cor1}).
The hypothesis that $n$ be large is used solely to show $|X|=1$ in the proof of Theorem \ref{thm-partitoincor}. If more information is known about $G$, $H$ and/or $|S|$, it may possible to combine this information with  Theorem \ref{cor1} to reduce how large $n$ must be to obtain the same conclusion. However, without such additional information, the bounds for $n\geq 2$  are optimal, as the following examples show. Since Theorem \ref{thm-partitoincor}(i) holds trivially for $n=1$, we do not worry about showing cases in Items 1--4 are optimal when the corresponding bound is $n\geq 2$ rather than $n\geq 1$.

For  the examples below, $H<G$ is a non-trivial, proper subgroup and $n\geq 1$ is an integer. We will define a subset $X\subseteq G/H$ with $nX$ aperiodic and $|nX|<n|X|$. We then set $Z=\phi_H^{-1}(X)\subseteq G$ and $S=\prod^{\bullet}_{g\in Z}g^{[n]}$. The sequence $S$ (with $S'=S$) will show the optimality of $n$ in Items 1--4.

\subsection*{Example A} Suppose $G/H=\la g\ra$ is a nontrivial cyclic group of order $|G/H|\geq 4$ and $n=\exp(G/H)-2=|G/H|-2$. Define $X=\{0,g\}$, $Z=\phi_H^{-1}(X)$ and $S=\prod^{\bullet}_{g\in Z}g^{[n]}$. Then $H=\mathsf H(\Sigma_n(S))$, \ $|S|=2|G|-4|H|$ and  $|\Sigma_n(S)|=|G|-|H|< 2|G|-4|H|-|G/H|+3=|S|-n+1$, but Theorem \ref{thm-partitoincor}(ii)(b) fails.

\subsection*{Example B} Suppose $G/H=(K/H)\oplus \la g\ra=(K/H)\oplus C_{\exp(G/H)}$ is a non-cyclic group with $\exp(G/H)\geq 3$ and $n=\exp(G/H)-1=|G/K|-1\geq 2$.  Define $X=(K/H)\cup\{g\}$, $Z=\phi_H^{-1}(X)$ and $S=\prod^{\bullet}_{g\in Z}g^{[n]}$. Then $H=\mathsf H(\Sigma_n(S))$, \ $|S|=(|G/K|-1)(|H|+|K|)$ and  $|\Sigma_n(S)|=|G|-|K|+|H|< |G|+(|H|-1)(|G/K|-1)-|K|+1=|S|-n+1$, but Theorem \ref{thm-partitoincor}(ii)(b) fails.

\subsection*{Example C} Suppose $G/H=(K/H)\oplus \la g\ra=(K/H)\oplus C_{\exp(G/H)}$ is a non-cyclic group with $|H|\geq n=\exp(G/H)\geq 2$ and $|K/H|\geq 3$.  Define $X=(K/H)\setminus \{0\}\cup\{g\}$, $Z=\phi_H^{-1}(X)$ and $S=\prod^{\bullet}_{g\in Z}g^{[n]}$. Since $|H/K|\geq 3$, we have  $H=\mathsf H(\Sigma_n(S))$. Moreover, $|S|=|G|$ and  $|\Sigma_n(S)|=|G|-|H|< |G|-n+1=|S|-n+1$, but Theorem \ref{thm-partitoincor}(ii)(b) fails.

\medskip

 If $H$ is trivial or equal to $G$, then Theorem \ref{thm-partitoincor}(i) necessarily holds, as (ii) requires $H$ to be proper and nontrivial. When $H$ is proper and nontrivial, then any of the following imply  one of Items 1--4 holds in Theorem \ref{thm-partitoincor}:
\begin{itemize}
\item[1.] $n\geq \exp(G)+1$,
\item[2.] $n\geq \exp(G)$ and $|G|<\exp(G)^2p$, where $p$ is the smallest divisor of $\frac{|G|}{\exp(G)}$ that is at least $3$,
\item[3.] $n\geq \exp(G)-1$ and  $G\cong C_p\times C_{\exp(G)}$ with $p$ prime,
\item[4.] $n\geq \frac1p|G|-1$ and $G$ is cyclic, where $p$ is the smallest prime divisor of $|G|$.
\end{itemize}
Thus Theorem \ref{thm-partitoincor} holds replacing Items 1--4 in Theorem \ref{thm-partitoincor} with Items 1--4 above. At the end of Section \ref{sec-setpartitions}, we also give a variation on Theorem \ref{thm-partitoincor}, namely Theorem \ref{thm-partitoincor-G}, where $|\Sigma_n(S)|\geq |S'|-n+1$ is replaced by $\Sigma_n(S)=G$, which corresponds to the case when we wish all elements of $G$ to be representable as $n$-term subsums of $S$. The necessary bounds for $n$ in this result are slightly better, since several sequences exhibiting the tightness of Items 1--4 above require $|S'|$ to be small in comparison to $|G|$, and can thus be eliminated when assuming $|S'|-n+1\geq |G|$.


\section{Prerequisites}\label{sec-prereq}

We begin by collecting together the main results to be used in the proof.  We begin with the following simple consequence of the pigeonhole principle \cite[Theorem 5.1]{Gbook}. Note, if $A$ and $B$ are each subsets of an $H$-coset with $|A|+|B|\geq |H|+1$, then the Pigeonhole Bound (applied to $A$ and $B$ translated so that they are subsets of the subgroup $H$) ensures that $A+B$ is an $H$-coset.

\begin{theirtheorem}[Pigeonhole Bound]\label{PigeonHoleBound} Let $G$ be an abelian group and let $A,\,B\subseteq G$ be finite subsets. If $|A|+|B|\geq |G|+r$ with $r\geq 1$ an integer, then $A+B=G$ with $\mathsf r_{A+B}(x)\geq r$ for every $x\in G$.
\end{theirtheorem}


The following two theorems give a  structural characterization of sumsets with $|nA|<n|A|$. The first is a special case of \cite[Corollary 3.2]{IttI}, while the second is \cite[Corollary 3.3]{IttI}.

\begin{theirtheorem}\label{cor1}
Let $G$ be a finite abelian group,  let $A\subseteq G$ be a nonempty subset with $\la A\ra_*=G$, let $n\geq 3$ be an integer, and let $K=\mathsf H(nA)$. If $n\geq \exp(G)+1$, then $|nA|\geq \min\{|G|,\,n|A|\}.$ If $n\geq \exp(G)-1$ and $|nA|< \min\{|G|,\,n|A|\}$, then one of the following holds.
\begin{itemize}

 \item[1.]     $n=\exp(G)$, \ $G=H\oplus \la g\ra\cong H\times C_{\exp(G)}$ with $K<H$, \ $|A|n\leq |G|$, \ $|G|-|K|=|nA|\geq |A+K|n-|K|$, and either
     \begin{itemize}
     \item[(a)] $H/K=H_1/K\oplus H_2/K\cong C_2^2$ and  $z+A+K=H_1\cup (g+H_2)$ for some $z\in G$, or
     \item[(b)]  $|H/K|\geq 3$ and $z+A+K=(H\setminus K)\cup (g+K)$ for some $z\in G$.
     \end{itemize}
\item[2.]     $n=\exp(G)-1$, \  $G\cong H\times C_{\exp(G)}$ with $K<H$ proper, $|\phi_H(A)|=2$, and either
\begin{itemize}
\item[(a)] $z+A+K=H\cup (A_0+K)$ for some $z\in G$ with $A_0=A\setminus H\neq \emptyset$, \  $|A|n\leq |G|$, and  $|nA|=|G|-|H|+|n(A_0+K)|$, or
\item[(b)] $G=H_0\oplus H_1\oplus\ldots \oplus H_r$ with  $K<H_0$ proper, $r\geq 1$ and $H_i=\la x_i\ra\cong C_{\exp(G)}$ for all $i\in [1,r]$, \ $z+A+K= \bigcup_{j=0}^{r}\big(K+\Sum{i=0}{j-1}H_i+\Sum{i=j+1}{r}x_i\big)$ for some $z\in G$, $|A|n\leq |G|-|H_0|+(\exp(G)-1)|K|\leq \frac{p\exp(G)^r+\exp(G)-p-1}{p\exp(G)^r}|G|$, where $p$ is the smallest prime divisor of $\exp(H_0)$, and $|nA|=|G|-|H_0|+|K|$.
\end{itemize}
\end{itemize}
\end{theirtheorem}

\begin{theirtheorem}\label{cor2}
Let $G$ be a finite abelian group,  let $A\subseteq G$ be a nonempty subset with $\la A\ra_*=G$, let $n\geq 1$ be an integer,  let $K=\mathsf H(nA)$ and suppose $n|A|>|G|.$
\begin{itemize}
\item[1.] If $n\geq \exp(G)$, then $nA=G$.
\item[2.] If $n=\exp(G)-1$ and $nA\neq G$, then $\exp(G)$ is composite, $G=H_0\oplus H_1\oplus\ldots \oplus H_r$ with  $K<H_0$ proper, $r\geq 1$ and $H_i=\la x_i\ra\cong C_{\exp(G)}$ for all $i\in [1,r]$ (thus $G$ is non-cyclic),  $$z+A+K= \bigcup_{j=0}^{r}\big(K+\Sum{i=0}{j-1}H_i+\Sum{i=j+1}{r}x_i\big)\quad\mbox{ for some $z\in G$},$$ $|A|n\leq |G|-|H_0|+(\exp(G)-1)|K|\leq \frac{p\exp(G)^r+\exp(G)-p-1}{p\exp(G)^r}|G|$, where $p$ is the smallest prime divisor of $\exp(H_0)$, and $|nA|=|G|-|H_0|+|K|$.
\end{itemize}
\end{theirtheorem}

The following lemma   combines with the Partition Theorem to show that only one of two extremes is possible for the subgroup $\la X\ra_*$.

\begin{lemma}\cite[Lemma 4.1]{IttI}\label{lem-partition-extra}
Let $G$ be an abelian group, let $n\geq 1$, let $S\in \Fc(G)$ be a sequence, let $S'\mid S$ be a subsequence  with $\mathsf h(S')\leq n\leq |S'|$, let $H=\mathsf H(\Sigma_n(S))$,   let $X\subseteq G/H$ be the subset of all $x\in G/H$ for which $x$ has multiplicity at least $n$ in $\phi_H(S)$, and let $Z=\phi_H^{-1}(X)$. Suppose $|\Sigma_n(S)|<|S'|-n+1$. Then either $$\la Z\ra_*=H\quad\mbox{ or }\quad\la Z\ra_*=\la \supp(S)\ra_*.$$
\end{lemma}

\section{Setpartitions}\label{sec-setpartitions}

We continue with the following technical lemma. Worth noting, the condition $\h(S)\leq n\leq |S|$ is a characterization of when there is a setpartition $\A=A_1\bdot\ldots\bdot A_n$ with $\mathsf S(\A)=S$; see \cite[Proposition 10.1]{Gbook}.

\begin{lemma}\label{lem-techPartitioning}
Let $G$ be a group, let $n\geq k\geq 1$ be integers, and let $S\in \Fc(G)$ be a sequence of terms from $G$ having a subsequence $S'\mid S$ with $\mathsf h(S')\leq n\leq |S'|$. Let $T\mid S$ be a maximal length subsequence such that $\mathsf h(T)\leq k\leq |T|$ and $|T|\leq |S'|-(n-k)$. Then there is a subsequence $T'\mid T^{[-1]}\bdot S$ such that $|T|+|T'|=|S'|$ and $\mathsf h(T')\leq n-k\leq |T'|$.
\end{lemma}

\begin{proof}
If $|T|=|S'|-(n-k)$, then $|T^{[-1]}\bdot S|=|S|-|S'|+(n-k)\geq n-k$, and the lemma follows taking $T'\mid T^{[-1]}\bdot S$ to be any subsequence of length $n-k\geq 0$. Therefore we may assume $|T|<|S'|-(n-k)$. In this case, the maximality of $T$ ensures that  $$\vp_x(T)=k\quad\mbox{ for any $x\in \supp(T^{[-1]}\bdot S)$},$$ else $T\bdot x$ would contradict the maximality of $T$.

The condition $|T|\leq |S'|-(n-k)\leq |S'|$ ensures that any subsequence $T'\mid T^{[-1]}\bdot S$ with $|T'|=|S'|-|T|\geq 0$ will have $n-k\leq |T'|$. Thus it suffices to show there is some subsequence $R\mid T^{[-1]}\bdot S$ with $\h(R)\leq n-k$ and $|R|\geq |S'|-|T|$, as then any subsequence $T'\mid R$ with $|T'|=|S'|-|T|$ will satisfy the conclusion of the lemma. To this end, consider a maximal length subsequence
$R\mid T^{[-1]}\bdot S$  with $\h(R)\leq n-k$. We must show  that $|R|\geq |S'|-|T|$.

In view of the maximality of $|R|$, we must have $\vp_x(R)= n-k$ for any $x\in \supp(R^{[-1]}\bdot T^{[-1]}\bdot  S)$, in which case $$\vp_x(T\bdot R)=\vp_x(T)+\vp_x(R)=k+(n-k)=n$$ for any  $x\in \supp(R^{[-1]}\bdot T^{[-1]}\bdot S)$. As a result, since $\h(S')\leq n$, we conclude that $$|S|-|S'|=|{S'}^{[-1]}\bdot S|\geq
\Summ{x\in G,\,\vp_x(S)\geq n}(\vp_x(S)-n)
\geq
|R^{[-1]}\bdot T^{[-1]}\bdot S|=|S|-|T|-|R|,$$ which implies the desired conclusion  $|R|\geq |S'|-|T|$, completing the proof.
\end{proof}

%
%
%

We are now ready to proceed with the proof of Theorem \ref{thm-partitoincor}.

\begin{proof}[Proof of Theorem \ref{thm-partitoincor}]
We begin by remarking that it suffices to prove Theorem \ref{thm-partitoincor} under the assumption
$\la \supp(S)\ra_*=G$. Indeed, suppose we know the theorem holds in this case. Then we can w.l.o.g. translate the terms of $S$ so that $0\in \supp(S)$ and $\la \supp(S)\ra=\la \supp(S)\ra_*$ and apply Theorem \ref{thm-partitoincor} using $\la \supp(S)\ra_*<G$ in place of $G$. Note if one of Items 1--4 holds, then one of Items 1--4 holds replacing $G$ by $\la \supp(S)\ra_*$, though it may not be the same item. If (ii) holds as result of the application of Theorem \ref{thm-partitoincor} using $\la \supp(S)\ra_*$, we are done, while if  $|\Sigma_n(S)|\geq |\Sum{i=1}{n}A_i|\geq |S'|-n+1$, then (i) follows. In the final case where $|\Sigma_n(S)|\geq |\Sum{i=1}{n}A_i|\geq |\la \supp(S)\ra_*|$, we conclude that $\Sigma_n(S)=\Sum{i=1}{n}A_i=\la \supp(S)\ra_*$, in which case (ii) is easily seen to hold using $H=K=\la \supp(S)\ra_*<G$ with $e_H=e_K=0$ and $\alpha=0$ (unless $\la \supp(S)\ra_*$ is trivial, in which case $\supp(S)=\{0\}$ and (i) holds). Thus we now assume $\la \supp(S)\ra_*=G$.

Let $H=\mathsf H(\Sigma_n(S))$, let $X\subseteq G/H$ be the subset of all $x\in G/H$ for which $x$ has multiplicity at least $n$ in $\phi_H(S)$, and let $Z=\phi_H^{-1}(X)\subseteq G$.
Apply Theorem \ref{thm-partition-thm} to $\Sigma_n(S)$ using $S'\mid S$ and let $\sA=A_1\bdot\ldots\bdot A_n$ be the resulting setpartition and $S''=\mathsf S(\mathscr A)$. Then $S''\mid S$ is a subsequence with $|S''|=|S'|$.  If $|\Sum{i=1}{n}A_i|\geq |S'|-n+1$, then (i) follows, as desired. Therefore assume Theorem \ref{thm-partition-thm}.2 holds.
Then, letting $N=|X|$ and $e_H=\Sum{i=1}{n}|A_i\setminus Z|$, it follows that $\Sigma_n(S)=\Sum{i=1}{n}A_i$ and  \ber \nn (|S'|-n+1)-(n-e_H-1)(|H|-1)+\rho&=& ((N-1)n+e_H+1)|H|\\ &\leq& |\Sum{i=1}{n}A_i|=|\Sigma_n(S)|\leq \min\{|G|-|H|,\,|S'|-n\},\label{gogag2}\eer  where $$\rho=N|H|n+e_H-|S'|\geq 0,$$ else the desired conclusion (i) holds.
In particular, $e_H\leq n-2$ and $H<G$ is proper and nontrivial. Thus the proof is complete when $|G|$ is prime, and we may proceed by induction on $|G|$. We also must have $N\geq 1$, else $e_H=|S'|$ follows, in which case  \eqref{gogag2} yields the contradiction  $|S'|-n\geq |\Sigma_n(S)|\geq (|S'|-n+1)|H|\geq |S'|-n+1$.
Thus $X$ is nonempty. Let \be\label{k-def-geq2}k=n-e_H\geq 2.\ee
By re-indexing the $A_i$, we can w.l.o.g. assume (in view of Theorem \ref{thm-partition-thm}.2) that \be\label{indexchoice}A_i+H=Z\quad\mbox{ for $i\in [1,k]$}\quad\und\quad |A_i\setminus Z|=1\mbox{ for $i\in [k+1,n]$}.\ee

\subsection*{Step A} $N=|X|=1$.

This step is the only place in the proof where the hypotheses in Items 1--4 will be used. We will show that any of these hypotheses, combined together with \eqref{gogag2} and Theorem \ref{cor1} applied to $nX$, forces $N=1$.

Assume by contradiction that  $N=|X|\geq 2$.
We have assumed  $\la \supp(S)\ra_*=G$, while $\la X\ra_*\leq G/H$ is nontrivial in view of $N=|X|\geq 2$, implying $H<\la Z\ra_*$. Consequently, in view of \eqref{gogag2}, and Lemma \ref{lem-partition-extra} applied to $S'\mid S$, we conclude that   \be \label{gob12}\la Z\ra_*=G\quad\und\quad\la X\ra_*=G/H.\ee
In view of \eqref{gogag2} and $H=\mathsf H(\Sigma_n(S))=\mathsf H(\Sum{i=1}{n}A_i)$, we have  \be\label{gob22}nX\neq G/H.\ee Since $e_H\leq n-2$ and $|S''|=|S'|$, we have \be\label{hawkeagel}|X||H|n=|Z|n\geq |S'|-e_H\geq |S'|-n+2.\ee In consequence, if $|nX|\geq |X|n$, then, since $X\subseteq \phi_H(A_i)$ for all $i\in [1,n]$ (per Theorem \ref{thm-partition-thm}.2), it follows that  $$|\Sum{i=1}{n}A_i|=|\Sum{i=1}{n}\phi_H(A_i)||H|\geq |nX||H|\geq |X||H|n>|S'|-n+1,$$ contrary to \eqref{gogag2}. Therefore we must instead have \be\label{gob33}|nX|<|X|n.\ee

If $n\geq \exp(G/H)+1$, then $n\geq 3$ as $H<G$ is proper, and then Theorem \ref{cor1} applied to $nX$ combined with  \eqref{gob12} and \eqref{gob22} implies that $|nX|\geq |X|n$, contrary to \eqref{gob33}.

If $n=\exp(G/H)>|H|$, then $n\geq 3$ as $H$ is nontrivial, and then Theorem \ref{cor1} applied to $nX$ combined with  \eqref{gob12}, \eqref{gob22} and \eqref{gob33} implies that \be\label{grapes}|G/H|-|K|=|nX|\geq |X+K|n-|K|\geq |X|n-|K|\ee
with one of the two possibilities listed in Theorem \ref{cor1}.1 holding, where $K=\mathsf H(nX)\leq G/H$.
If $|(X+K)\setminus X|\geq \frac12|K|$, then \eqref{grapes} implies that $|nX|\geq |X|n$, contrary to \eqref{gob33}. Therefore, we must have $|(X+K)\setminus X|\leq  \frac{|K|-1}{2}$, in which case Theorem \ref{PigeonHoleBound} implies that $2X$ is $K$-periodic.  In view of \eqref{indexchoice} and \eqref{k-def-geq2}, we have $\phi_H(A_1)+\phi_H(A_2)=2X$, which is $K$-periodic as just noted. Consequently,  if $\Sum{i=1}{n}\phi_H(A_i)\neq nX$, then $|\Sum{i=1}{n}\phi_H(A_i)|\geq |nX|+|K|\geq |X|n$, with the last inequality from \eqref{grapes}. In this case, \eqref{hawkeagel} yields $|\Sum{i=1}{n}A_i|=|\Sum{i=1}{n}\phi_H(A_i)||H|\geq |X||H|n> |S'|-n+1$, contradicting \eqref{gogag2}.
Therefore we instead conclude that
\be\label{returnper}\Sum{i=1}{n}\phi_H(A_i)= nX.\ee Since $\Sigma_n(S)=\Sum{i=1}{n}A_i$ with $\mathsf H(\Sigma_n(S))=H$ (in view of Theorem \ref{thm-partition-thm}.2), it follows that $\Sum{i=1}{n}\phi_H(A_i)= nX$ is aperiodic, meaning $K$ is trivial.
If $e_H=0$, then $|S'|\leq |X||H|n$, whence \eqref{grapes} and \eqref{returnper} imply $|\Sum{i=1}{n}A_i|=|\Sum{i=1}{n}\phi_H(A_i)||H|= |nX||H|\geq |X||H|n-|H|\geq |S'|-|H|\geq |S'|-n+1$, with the final inequality in view of $n=\exp(G/H)>|H|$. However this contradicts \eqref{gogag2}. Therefore we can instead assume $e_H\geq 1$. Thus, in view of \eqref{indexchoice}, there is some $z\in A_n$ such that $\phi_H(z)\notin X$.
However, examining both possibilities for the set $X$ given in Theorem \ref{cor1}.1 with $K$ trivial, we find that $(n-1)X+(X\cup \{y\})=G/H$ whenever $y\notin X$, where $n=\exp(G/H)\geq 3$. Thus, since $\phi_H(A_n)=X\cup \{\phi_H(z)\}$ with $\phi_H(z)\notin X$, we conclude that  $G/H=(n-1)X+(X\cup \{\phi_H(z)\})\subseteq \Sum{i=1}{n}\phi_H(A_i)$, implying $\Sum{i=1}{n}A_i=G$.
However this once again contradicts \eqref{gogag2}.

If $n\geq \exp(G/H)$ and $G/H\cong C_2\times C_{\exp(G/H)}$, then either $n\geq 3$, in which case  Theorem \ref{cor1} applied to $nX$ combined with  \eqref{gob12} and \eqref{gob22}  implies that $|nX|\geq |X|n$, contrary to \eqref{gob33}, or else $n=2=\exp(G/H)$, in which case $G/H\cong C_2^2$. In the latter case, \eqref{gob12} forces $|X|\geq 3$, in which case Theorem \ref{PigeonHoleBound} implies $nX=2X=G/H$, contrary to \eqref{gob22}.

If $n\geq \exp(G/H)-1$ and $G/H$ is cyclic, then $|G/H|=\exp(G/H)$. If $n=1$, then (i) trivially holds, so we may assume $n\geq 2$. If $n=2$, then $|G/H|\leq 3$. In such case, Theorem \ref{PigeonHoleBound} applied to $X$ implies that $2X=nX=G/H$, contradicting \eqref{gob22}. Therefore we can instead assume $n\geq 3$, and then Theorem \ref{cor1} applied to $nX\subseteq G/H$ combined with  \eqref{gob12}, \eqref{gob22} and $G/H$ cyclic implies that $|nX|\geq |X|n$, contradicting \eqref{gob33}. As this exhausts all possibilities for the hypotheses given in Items 1--4, Step A is complete.

 \bigskip

In view of Claim A, \eqref{gogag2} now yields \be\label{wrench}(n-k+1)|H|=(e_H+1)|H|\leq |\Sum{i=1}{n}A_i|=|\Sigma_n(S)|\leq\min\{|G|-|H|,\,|S'|-n\},\ee in turn implying $e_H\leq \min\{|G/H|-2,\,\frac{|S'|-n}{|H|}-1\}$.
Indeed, it is easily observed that both upper bounds for $e_K$ and $e_H$ in (ii)(b) follow from the respective lower bound for $|\Sigma_n(S)|$ combined with \eqref{gogag2}.
Note \eqref{wrench} implies that \be\label{nail} |S'|\geq (e_H+1)|H|+n=(n-k+1)|H|+n.\ee In view of Step A, there is some $\alpha\in G$ such that $Z=\alpha+H$ and  $e_H$ is the number of terms of $S$ lying outside the coset $\alpha+H$.  Since $H<G$ is proper and $\la \supp(S)\ra_*=G$, we must have $e_H\geq 1$, and thus $2\leq k\leq n-1$, implying $n\geq 3$.

By translating all terms of $S$ appropriately, we can w.l.o.g. assume  $\alpha=0$ and $Z=H$.
In view of \eqref{indexchoice} and  Theorem \ref{thm-partition-thm}.2, we have    $A_i\cap H\neq \emptyset$ for all $i\in [1,n]$,  \ $\Sigma_n(A)=\Sum{i=1}{n}A_i$,  \ $\supp(\mathsf S(\sA)^{[-1]}\bdot S)\subseteq H$,  \be\label{index-N=1setup}A_i\subseteq H\quad\mbox{ for all $i\in [1,k]$,} \quad\und\quad A_i\setminus H=\{z_i\}\quad\mbox{ for all $i\in [k+1,n]$},\ee
for some $z_i\in G\setminus H$.


Let $S_H\mid S$ be the subsequence of $S$ consisting of all terms from $H$ and let $S'_H\mid S''$ be the subsequence of $S''$ consisting of all terms from $H$. Then $|S_H|=|S|-e_H$ and $|S'_H|=|S'|-e_H$ (recall $|S'|=|S''|$). Moreover, $\h(S'_H)\leq n\leq |S'_H|$ since $A_i\cap H\neq\emptyset $ for all $i\in [1,n]$ and $S'_H\mid S''=\mathsf S(\A)$.
Let $T\mid S_H$ be a maximal length subsequence such that $\h(T)\leq k\leq |T|$ and $|T|\leq |S'_H|-(n-k)$. By translating all terms of $S$ by an appropriate element from $H$, we can w.l.o.g. assume $0\in \supp(S_H)$. Let \be\label{defH'}G'=\la \supp(S_H)\ra\leq H \quad \und\quad H'=\mathsf H(\Sigma_k(S_H))\leq G'.\ee
In view of the definition of $T$ and \eqref{nail}, we have
\be\label{quartz}|T|\geq \frac{k}{n}|S'_H|=\frac{k}{n}(|S'|-e_H)\geq \frac{k}{n}\Big((n-k+1)|H|+k\Big).\ee

\subsection*{Step B} If there exists a setpartition $\mathscr B=B_1\bdot\ldots\bdot B_k\in \sP(H)$ with $\mathsf S(\mathscr B)\mid S_H$, \ $|\mathsf S(\mathscr B)|=|T|$ and $\Sum{i=1}{k}B_i=G'$, then (ii) holds taking $K=G'$.

Recall that $S_H\mid S$ is the subsequence of all terms from $H$ and that we have assumed $0\in \supp(S_H)$ with $\la \supp(S_H)\ra=G'\leq H$. This means that a term of $S$ is from $H$ if and only if it is from $G'\leq H$. Thus, setting $K=G'$, we find that $e_H=e_K$ is also the number of terms of $S$ lying outside the subgroup $G'$.

Let $R=\mathsf S(\mathscr B)$. Since $\mathscr B$ is a $k$-setpartition, our hypotheses give  $\h(R)\leq k\leq |R|=|T|\leq |S'_H|-(n-k)$, with the latter inequality in view of the definition of $T$. In view of Lemma \ref{lem-techPartitioning} (applied taking $T$ to be $R$, taking $S$ to be $S_H$, and taking $S'$ to be $S'_H$), it follows that there is a subsequence $R'\mid R^{[-1]}\bdot S_H$ such that $|R|+|R'|=|S'_H|=|S'|-e_H$ and $\h(R')\leq n-k\leq |R'|$. The latter is equivalent  to there existing a setpartition $\mathscr B'=B_{k+1}\bdot\ldots\bdot B_n$ with $\mathsf S(\mathscr B')=R'$. Define $\mathscr C=C_1\bdot\ldots\bdot C_n=B_1\bdot\ldots\bdot B_k\bdot (B_{k+1}\cup \{z_{k+1}\})\bdot\ldots \bdot (B_{n}\cup \{z_{n}\})\in \sP(G)$. Then $\mathsf S(\mathscr C)\mid S$ with $|\mathsf S(\mathscr C)|=|R|+|R'|+e_H=|S'|$.

The hypothesis of Step B ensures that (ii)(d) holds for $\mathscr C$.  Since a term of $S$ is from $H$ if and only if it is from $G'$, and since $\mathsf S(\mathscr C)^{[-1]}\bdot S\mid S_H$ (as the $z_i$ are precisely those terms of $S$ lying outside the subgroup $H$), it follows that $\supp(\mathsf S(\mathscr C)^{[-1]}\bdot S)\subseteq G'=K$ and that (ii)(c) holds for $\mathscr C$.  Since $K=G'\leq H$, we also have $|\Sigma_n(S)|\geq (e_H+1)|H|=(e_K+1)|H|\geq (e_K+1)|K|$, whence (ii)(b) holds in view of \eqref{gogag2}.
 It remains to show $\Sum{i=1}{n}C_i=\Sigma_n(S)$. To this end, since $\Sum{i=1}{n}A_i=\Sigma_n(S)$ (in view of Theorem \ref{thm-partition-thm}.2 holding for $\A$), it suffices to show $\Sum{i=1}{n}A_i\subseteq \Sum{i=1}{n}C_i$ to complete the step (as the reverse inclusion $\Sum{i=1}{n}C_i\subseteq \Sigma_n(S)=\Sum{i=1}{n}A_i$ is trivial). Let $x=a_1+\ldots+a_n\in \Sum{i=1}{n}A_i$ with $a_i\in A_i$ for $i\in [1,n]$ be arbitrary.  Let $I\subseteq [1,n]$ be all indices $i\in [1,n]$ with $a_i\notin H$. Then $a_i=z_i$ for all $i\in I$, while $a_i\in K$ for all $i\in [1,n]\setminus I$. Hence $x\in \Summ{i\in I}z_i+K$. However, since $\Sum{i=1}{k}C_i=\Sum{i=1}{k}B_i=K$ with all  $z_i\in C_i$ for $i\in [k+1,n]$, it follows that $x\in \Summ{i\in I}z_i+K\subseteq \Sum{i=1}{n}C_i$. Consequently, since $x\in \Sum{i=1}{n}A_i$ was arbitrary, we conclude that $\Sum{i=1}{n}A_i\subseteq \Sum{i=1}{n}C_i$. Thus (ii)(a) also holds for $\mathscr C$, completing Step B.

\subsection*{Step C}  $5\leq k\leq n-2$

If $\Sum{i=1}{k}A_i=H$, then (ii) follows taking $K=H$ and $\alpha=0$, and the proof is complete. Therefore, we may instead assume
$\Sum{i=1}{k}A_i\neq  H.$
%
%
As a result, if $e_H=1$, then $\Sum{i=1}{n}A_i\subseteq \{0,z_n\}+H$ with $\Big(\Sum{i=1}{n}A_i\Big)\cap (z_n+H)=z_n+\Sum{i=1}{n-1}A_i=z_n+\Sum{i=1}{k}A_i\neq z_n+H$, contradicting that $H=\mathsf H(\Sigma_n(S))=\mathsf H(\Sum{i=1}{n}A_i)$. Therefore, we instead conclude that $e_H\geq 2$, and thus \be\label{supper} 2\leq k=n-e_H\leq n-2.\ee In particular, $n\geq 4$.


Since $\h(T)\leq k\leq |T|$, there is a setpartition  $\mathscr A'=A'_1\bdot\ldots\bdot A'_k$ with $\mathsf S(\mathscr A')=T$, and we can w.l.o.g assume $|A'_1|\geq |A'_2|\geq \ldots\geq |A'_k|$. Then, since $k\geq 2$, it follows by the Pigeonhole Principle that  \be\label{twinden}|A'_1|+|A'_2|\geq \frac{2}{k}|T|.\ee
If  $|A'_1|+|A'_2|\geq |H|+1$,  then Theorem \ref{PigeonHoleBound} implies that $A'_1+A'_2=H$, forcing $H=G'$, whence Step B completes the proof in view of $k\geq 2$. Therefore we may instead assume \be\label{horsecart} |A'_1|+|A'_2|\leq |H|.\ee


The definition of $\rho$ combined with Step A gives $\rho=|H|n-|S'_H|$, and thus the first inequality in \eqref{quartz} yields $|T|\geq k|H|-\frac{k}{n}\rho$. Combined with \eqref{horsecart} and \eqref{twinden}, we find that $|H|\geq |A'_1|+|A'_2|\geq \frac2k|T|\geq 2|H|-\frac{2}{n}\rho$, implying $\rho\geq\frac12n|H|$. On the other hand, \eqref{gogag2} ensures that $\rho<(k-1)(|H|-1)$. Thus \be\label{trollack}\frac12n|H|\leq \rho<(k-1)(|H|-1).\ee Combining \eqref{supper} and \eqref{trollack} gives the desired bounds $5\leq k\leq n-2$, completing Step C.

\subsection*{Step D} If there is a setpartition $\mathscr B=B_1\bdot\ldots\bdot B_k\in \sP(H)$ with $\mathsf S(\mathscr B)\mid S_H$, $|\mathsf S(\mathscr B)|=|T|$ and $|\Sum{i=1}{k}B_i|\geq \min\{|G'|,\,|T|-k+1\}$, then (ii) holds taking $K=G'$.


In view of \eqref{quartz}, we have \be\label{tryon}|T|-k+1\geq \frac{k(n-k+1)|H|-k(n-k)}{n}+1.\ee The right hand side of \eqref{tryon}  is quadratic in $k$ with negative lead coefficient (since $H$ is nontrivial), thus minimized at a boundary value for $k$. In view of Step C, we have $k\in [3,n-2]$ with $n\geq 6$. Thus the bound in \eqref{tryon} is minimized for $k=3$, yielding  $|T|-k+1>\frac{3(n-2)}{n}|H|-2\geq 2|H|-2\geq |H|\geq |G'|$. Consequently, since $0\in \supp(S_H)$ with $\la \supp(S_H)\ra=G'$, it follows that the hypotheses of Step D yield $\Sum{i=1}{k}B_i=G'$, and now Step B yields (ii) taking $K=G'$, completing Step D.


\subsection*{Step E} $k> |H/H'|+2$.

Recall that  $H'=\mathsf H(\Sigma_k(S_H))\leq G'\leq H$ as defined in \eqref{defH'}.
Apply Theorem \ref{thm-partition-thm} to  $T\mid S_H$ and $\Sigma_k(S_H)$. Then, in view of Step D, we can assume Theorem \ref{thm-partition-thm}.2 holds. Let $\mathscr B=B_1\bdot\ldots\bdot B_k$ be the resulting setpartition. Theorem \ref{thm-partition-thm}.2 yields $$|H|-|H'|\geq |\Sigma_k(S_H)|= |\Sum{i=1}{k}B_i|\geq |T|-(k-1)|H'|,$$ with the upper bound holding in view of Step D and $G'\leq H$. Thus \be\label{wallyjungle}k\geq \frac{|T|-|H|}{|H'|}+2.\ee  In view of \eqref{quartz} and Step C, we have $|T|> \frac{k(n-k+1)}{n}|H|\geq \frac{3(n-2)}{n}|H|>2|H|$, which combined with \eqref{wallyjungle} yields the desired bound for $k$, completing Step E.

\bigskip

Since $H<G$ is proper and $\h(T)\leq k\leq |T|\leq |S'_H|-(n-k)$ (in view of the definition of $T$), it follows from Step E and \eqref{defH'} that we can apply the induction hypothesis to $T\mid S_H$ and $\Sigma_k(S_H)$. Let $\mathscr B=B_1\bdot\ldots\bdot B_k\in \sP(G')$ be the resulting setpartition and let $R=\mathsf S(\mathscr B)\mid S_H$. Then $|R|=|\mathsf S(\mathscr B)|=|T|$ and, in view of Step D, we can assume Theorem \ref{thm-partitoincor}(ii) holds for $\Sigma_k(S_H)$ with nontrivial subgroup $K\leq H'< G'\leq H$ and $\alpha'\in G'\leq H$.
Let $S_K\mid S$ be the subsequence of all terms of $S$ from $\alpha'+K$. By translating all terms of $S$ appropriately, we can w.l.o.g. assume $\alpha'=0$ and $0\in \supp(S_K)$.
 Let \be\label{batgopher}e'_K\leq \min\{|G'/K|-2,\frac{|T|-k}{|K|}-1\}\ee be the number of terms of $S_H$ lying outside the subgroup $K$, and let $k'=k-e'_K=n-e_H-e'_K=n-e_K$, where $$e_K=e_H+e'_K$$ is the number of terms of $S$ lying outside the subgroup $K$.

 In view of Lemma \ref{lem-techPartitioning} (applied taking $T$ to be $R$, taking $S$ to be $S_H$, and taking $S'$ to be $S'_H$), it follows that there is a subsequence $T'\mid  R^{[-1]}\bdot S_H$ such that $|R|+|T'|=|T|+|T'|=|S'_H|=|S'|-e_H$ and $\h(T')\leq n-k\leq |T'|$. The latter is equivalent  to there existing a setpartition $\mathscr B'=B_{k+1}\bdot\ldots\bdot B_n\in \sP(G')$ with $\mathsf S(\mathscr B')=T'$.  Then $$\mathscr C=C_1\bdot\ldots\bdot C_n=B_1\bdot\ldots\bdot B_k\bdot (B_{k+1}\cup \{z_{k+1}\})\bdot\ldots
\bdot (B_n\cup \{z_n\})\in \mathsf S(G)$$ is a setpartition with $\mathsf S(\mathscr C)\mid S$ and $|\mathsf S(\mathscr C)|=|T|+|T'|+e_H=|S'|$. We will show that (ii) holds taking $\sA$ to be $\mathscr C$ and taking $K$ as defined above.

Since (ii)(d) holds for $\mathscr B=B_1\bdot\ldots\bdot B_k$, we have $\Sum{i=1}{k'}C_i=\Sum{i=1}{k'}B_i=K$, whence (ii)(d) holds for $\mathscr C$.
 Since (ii)(c) and (ii)(a) hold for $\mathscr B=B_1\bdot\ldots\bdot B_k$, it follows that $K\cap C_i=K\cap B_i\neq \emptyset$ for all $i\in [1,k]$, that $C_i=B_i\subseteq K$ for all $i\in [1,k']$, that $C_i\setminus K=B_i\setminus K=\{z_i\}$ for all $i\in [k'+1,k]$, for some $z_i\in G'\setminus K$, and that $$\supp(R^{[-1]}\bdot S_H)=\supp(\mathscr S(B_1\bdot\ldots\bdot B_k)^{[-1]}\bdot S_H)\subseteq K.$$ Thus $B_i\subseteq K$ for all $i\in [k+1,n]$ and $\supp(\mathsf S(\mathscr C)^{[-1]}\bdot S)\subseteq K$, in which case (ii)(c) holds for $\mathscr C$. Since (ii)(b) holds for $\Sigma_k(S_H)$, we have $(e'_K+1)|K|\leq |\Sigma_k(S_H)|\leq |G'|\leq |H|$. Thus \eqref{wrench} implies \ber\nn |\Sigma_n(S)|&\geq& (e_H+1)|H|=(e_K-e'_K+1)|H|\geq (e_K-e'_K+1)(e'_K+1)|K|\\\nn &=&(e_K+e'_Ke_K-{e'_K}^2+1)|K|\geq (e_K+1)|K|,\eer in which case (ii)(b) holds for $\mathscr C$ in view of \eqref{gogag2}.
 It remains to show $\Sum{i=1}{n}C_i=\Sigma_n(S)$. To this end, since $\Sum{i=1}{n}A_i=\Sigma_n(S)$ (in view of Theorem \ref{thm-partition-thm}.2 holding for $\A$), it suffices to show $\Sum{i=1}{n}A_i\subseteq \Sum{i=1}{n}C_i$ to complete the proof (as the reverse inclusion $\Sum{i=1}{n}C_i\subseteq \Sigma_n(S)=\Sum{i=1}{n}A_i$ is trivial). Let $x=a_1+\ldots+a_n\in \Sum{i=1}{n}A_i$ with $a_i\in A_i$ for $i\in [1,n]$ be arbitrary.
 Since (ii)(c) holds for $\mathscr B=B_1\bdot\ldots\bdot B_k$, it follows that, for every $i\in [k'+1,k]$,  we have $C_i\setminus K=B_i\setminus K=\{z_i\}$ for some $z_i\in G'\setminus K$. Moreover the terms $z_{k'+1},z_{k'+2},\ldots,z_n$ are precisely those terms of $S$ lying outside the subgroup $K$.
Thus  $x\in \Big(\Sigma(z_{k'+1}\bdot\ldots\bdot z_n)\cup \{0\}\Big)+K$.
However, since $\Sum{i=1}{k'}B_i=\Sum{i=1}{k'}C_i=K$ (in view of (ii)(d) holding for $\mathscr B=B_1\bdot\ldots\bdot B_k$), since each $z_i\in C_i$ for $i\in [k'+1,n]$,  and since $K\cap C_i\neq \emptyset $ for all $i\in [1,n]$, it follows that $x\in \Big(\Sigma(z_{k'+1}\bdot\ldots\bdot z_n)\cup \{0\}\Big)+K\subseteq \Sum{i=1}{n}C_i$. Consequently, since $x\in \Sum{i=1}{n}A_i$ was arbitrary, we conclude that $\Sum{i=1}{n}A_i\subseteq \Sum{i=1}{n}C_i$, and now (ii)(a) holds for $\mathscr C$, completing the proof.
\end{proof}

We conclude with the following variation on Theorem \ref{thm-partitoincor}. As was the case for Theorem \ref{thm-partitoincor}, any of the following conditions combined with $H<G$ being proper and nontrivial ensures that one of Items 1--3 holds in Theorem \ref{thm-partitoincor-G}, and thus they can be substituted for Items  1--3 in Theorem \ref{thm-partitoincor-G}.

\begin{itemize}
\item[1.] $n\geq \exp(G)$, or
\item[2.] $n\geq \exp(G)-1$ and $G\cong K\times C_{\exp(G)}$ with either $\exp(G)$ or $|K|$ prime, or
\item[3.] $n\geq \frac1p |G|-1$ and $G$ is cyclic, where $p$ is the smallest prime divisor of $|G|$, or
\item[4.] $n\geq 1$ and either $\exp(G)\leq 3$ or $|G|<10$.

\end{itemize}

While the bounds given in Items 1--3 of Theorem \ref{thm-partitoincor-G} below are not tight, the worse-case scenario ones given in Items 1--4 above are, as can be seen by Examples B.1--B.3 in \cite{IttI}. It seems to be a more challenging problem to find optimal bounds in Theorem \ref{thm-partitoincor-G} for $n$ in terms of $G/H$, rather than $G$, particularly when $G/H$ is not close to cyclic.

\begin{theorem}
\label{thm-partitoincor-G}
Let $G$ be a finite abelian group, let $n\geq 1$, let $S\in \Fc(G)$ be a sequence of terms from $G$ with  $H=\mathsf H(\Sigma_n(S))$, and let   $S'\mid S$ be a subsequence with $\mathsf h(S')\leq n\leq |S'|$ and $|S'|\geq n+|G|-1$. Suppose either $H$ is trivial, equal to $G$, or that one of the following holds:
\begin{itemize}
\item[1.]  $n\geq \exp(G/H)$, or
\item[2.] $n\geq \exp(G/H)-1$ and either $G/H$ is cyclic or $\exp(G/H)$ is prime, or
\item[3.] $n\geq 1$ and either $\exp(G/H)\leq 3$ or $G/H\cong C_4$.
\end{itemize}
Then there is a setpartition $\mathscr A=A_1\bdot\ldots\bdot A_n\in \sP(G)$ with $\mathsf S(\mathscr A)\mid S$ and $|\mathsf S(\mathscr A)|=|S'|$ such that either $\Sigma_n(S)=\Sum{i=1}{n}A_i=G$ or else Theorem \ref{thm-partitoincor}(ii) holds.
\end{theorem}

\begin{proof}
The proof is a minor variation on that of Theorem \ref{thm-partitoincor} combined with the hypothesis $|S'|\geq |G|+n-1$, which ensures $|S'|-n+1\geq |G|$. We only highlight the few differences. First assume $\la \supp(S)\ra*=G$ and
let all notation be is in the proof of Theorem \ref{thm-partitoincor}, including $H$, $X$, $\mathscr A=A_1\bdot\ldots\bdot A_n$, $N$, and $e_H$.
If $n=1$, then the hypotheses  $\mathsf h(S')\leq n$ and $|S'|\geq n+|G|-1=|G|$ ensure that $\Sigma_n(S')=\supp(S')=G$, as desired. If $n=2$, then $|S'|\geq n+|G|-1=|G|+1$, in which case Theorem \ref{PigeonHoleBound} implies $A_1+A_2=G$, and thus $\Sigma_2(S)=G$, as desired. Therefore we can assume $n\geq 3$.
Since $N|H|n+e_H\geq |S'|\geq |G|+n-1$ and $e_H<n-1$,  it follows that $|X|n=Nn>|G/H|$.
 Thus in Step A we may use Theorem \ref{cor2} instead of Theorem \ref{cor1} applied to $nX$, in which case the hypotheses in Item 1 or Item 2 are enough to secure the contradiction $nX=G/H$. Moreover, the hypothesis in Item 3 combined with $n\geq 3$ ensures that either Item 1 or 2 holds. This allows us to conclude $|X|=N=1$. The rest of the proof is now identical to that of Theorem \ref{thm-partitoincor}. Note that $|T|\geq |H|+k-1\geq |G'|+k-1$ is shown in Step D, allowing us to apply the induction hypothesis to $T$ after Step E. This shows the theorem to hold when $\la \supp(S)\ra_*=G$. When $\la \supp(S)\ra_*<G$, then we can translate the terms of $S$ so that $0\in \supp(S)$ and apply Theorem \ref{thm-partitoincor-G} using $\la\supp(S)\ra=\la \supp(S)\ra_*$ instead of $G$. If Theorem \ref{thm-partitoincor}(ii) holds, we are done, while if $\Sigma_n(S)=\Sum{i=1}{n}A_i=\la \supp(S)\ra_*$, then Theorem \ref{thm-partitoincor}(ii) is easily seen to hold with $K=H=\la \supp(S)\ra_*$ and $e_H=e_K=0$, unless $\la \supp(S)\ra_*$ is trivial. However, if $\la \supp(S)\ra_*$ is trivial, then $\supp(S)=\{0\}$ and $\mathsf h(S')=|S'|=n$. Hence  $n=|S|'\geq n+|G|-1$ implies $|G|$ is trivial, and now $\Sum{i=1}{n}A_i=G=\{0\}$ follows, as desired. Thus the theorem follows in the case $\la\supp(S)\ra_*<G$ as well.
\end{proof}

We remark that it would be interesting to know whether it is always possible to take $K=H$ in Theorem \ref{thm-partitoincor}. Related to whether this is true or not is the question of whether there are examples of cardinality two subsets $A_1,\ldots,A_n\subseteq G$ such that $A_1+\ldots+A_n$ is aperiodic and there does not exist any $x\in A_1+\ldots+A_n$ with $\mathsf r_{A_1+\ldots+A_n}(x)=1$, where $\mathsf r_{A_1+\ldots+A_n}(x)$ denotes the number of tuples $(a_1,\ldots,a_n)\in A_1\times\ldots\times A_n$ with $a_1+\ldots+a_n=x$.


\begin{thebibliography}{99}

\bibitem{bialostocki-I} A. Bialostocki and P. Dierker, On the Erd\H{o}s-Ginzburg-Ziv theorem and the Ramsey
numbers for stars and matchings, \emph{Discrete Math.}, \textbf{110} (1992), no. 1–3, 1–-8.

\bibitem{Bialostocki-II} A. Bialostocki and M. Lotspeich, Developments of the Erd\H{o}s-Ginzburg-Ziv Theorem I, in
Sets, graphs and numbers,
\emph{Colloq. Math. Soc. J\'anos Bolyai} \textbf{60} (1992), North-Holland,
Amsterdam, 1992, 97-–117


\bibitem{bialostock-nondecreasing} A. Bialostocki, P. Erd\H{o}s and H. Lefmann,  Monochromatic and zero-sum sets of nondecreasing diameter,
\emph{Discrete Math.}
\textbf{137} (1995),
no. 1-3, 19-–34.


\bibitem{arie-II} A. Bialostocki, P. Dierker, D. J. Grynkiewicz and M. Lotspeich,  On some developments of the Erd\H{o}s-Ginzburg-Ziv theorem II, \emph{Acta Arith.}  \textbf{110} (2003), no. 2, 173--184.

\bibitem{bollobas-leader} B. Bollob\'as and I. Leader, The number of
$k$-sums modulo $k$,
\emph{J. Number Theory} \textbf{78} (1999),  27-–35.


\bibitem{Cao} H. Q. Cao, An Addition Theorem on the Cyclic Group $\Z/p^\alpha q^\beta\Z$, \emph{Electronic J. Combin.} \textbf{13} (2006) (electronic).

\bibitem{Caro-WEGZ-survey} Y. Caro, Zero-sum problems—--a survey,
\emph{Discrete  Math.}
\textbf{152} (1996), no. 1–-3, 93–-113.


\bibitem{DGM} M. DeVos, L. Goddyn and B. Mohar,
A generalization of Kneser’s addition theorem,
\emph{Adv. Math.} \textbf{220} (2009), no. 5, 1531-–1548.


\bibitem{egz} P. Erd\H{o}s, A. Ginzburg, A. Ziv, Theorem in Additive
Number Theory, \emph{Bull. Res. Council Israel} \textbf{10F} (1961), 41--43.

\bibitem{graham-original} P. Erd\H{o}s and R. L. Graham, Old and new problems and results in combinatorial number theory, Monographies
de  L’Enseignement  Mathmatique  [Monographs  of  L’Enseignement  Mathmatique] \textbf{28} (1980), Universit\'e de Gen\'eve,
L’Enseignement Math\'ematique, Geneva,  pp. 95.

\bibitem{grahamconj} P. Erd\H{o}s and E. Szemer\'edi, On a problem of Graham,
\emph{Publ. Math. Debrecen}, \textbf{23} (1976), no. 1--2, 123–-127.

\bibitem{Furedi-kleitman-m-term-zs} Z.  F\H{u}redi  and  D.  J.  Kleitman,
The  minimal  number  of  zero  sums,
in  Combinatorics,  Paul
Erd\H{o}s is eighty \textbf{1},  Bolyai Soc. Math. Stud.,  J\'anos Bolyai Math. Soc. (1993),  Budapest,
159–-172.

\bibitem{Gao-preolson} W. Gao,
Addition theorems for finite abelian groups,
\emph{J. Number Theory} \textbf{53} (1995), no.2, 241-–246.


 \bibitem{gao-ger-survey} W. Gao and A. Geroldinger, Zero-sum problems in finite abelian groups:  A survey,
\emph{Expositiones Mathematicae}
\textbf{24} (2006), no. 4, 337-–369.


\bibitem{PropB} W. Gao, A. Geroldinger and D. J. Grynkiewicz, Inverse Zero-Sum Problems III, \emph{Acta Arithmetica}
\textbf{141} (2010), 103--152.

\bibitem{Gao-conj-ordaz-paper} W. D.  Gao,  R.  Thangadurai  and  J.  Zhuang,  Addition  theorems  on  the  cyclic  groups  of  order
$p^l$, \emph{Discrete Mathematics} \textbf{308}
(2008), no. 10, 2030--2033.

\bibitem{Gao-nsum-paper} W. Gao, D. J. Grynkiewicz and X. Xia, On $n$-Sums in an Abelian Group, \emph{Combin. Prob. and Comput.} \textbf{25} (2016), no. 3, 419--435.

\bibitem{Gao-bonus} W.  Gao, Y.  Li, P.  Yuan and J.  Zhuang, On the structure of long zero-sum free sequences and $n$-zero-sum free sequences over finite cyclic groups, \emph{Arch. Math. (Basel)} \textbf{105} (2015), no. 4, 361–-370.

\bibitem{Alfred-Ruzsa-book} A. Geroldinger and I. Ruzsa, \emph{Combinatorial Number Theory and Additive Group Theory}, Birkh\H{a}user (2009), Basel.

\bibitem{alfredbook} A. Geroldinger and F. Halter-Koch, \emph{Non-Unique Factorizations: Algebraic, Combinatorial and Analytic Theory}, Pure and Applied Mathematics: A series of Monographs and Textbooks \textbf{278} (2006),  Chapman \& Hall, an imprint of Taylor \& Francis Group, Boca Raton, FL.



\bibitem{GG-nonabelian-index2} A. Geroldinger and D. J. Grynkiewicz, The Large Davenport Constant I: Groups with a Cyclic, Index 2 Subgroup, \emph{J. Pure and Applied Alg.}
\textbf{217} (2013), no. 5, 863--885.

\bibitem{nondescrea-diam-zerosum} D. J. Grynkiewicz,  On four colored sets with nondecreasing diameter and the Erd\H{o}s-Ginzburg-Ziv Theorem,  \emph{J. Combin. Theory, Ser. A}   \textbf{100} (2002), no. 1, 44--60.

\bibitem{ccd} D. J. Grynkiewicz, On a partition analog of the Cauchy-Davenport theorem,  \emph{Acta Math. Hungar.}  \textbf{107} (2005), no. 1--2, 161--174.



\bibitem{hamconj} D. J. Grynkiewicz, On a conjecture of Hamidoune for subsequence sums, \emph{Integers}  \textbf{5} (2005), no. 2, A7   (electronic).

\bibitem{hypergraph-egz} D. J. Grynkiewicz, An extension of the Erd\H{o}s-Ginzburg-Ziv theorem to hypergraphs, \emph{European J. Combin.} \textbf{26} (2005), no. 8, 1154--1176.




\bibitem{wegz} D. J. Grynkiewicz,  A weighted Erd\H{o}s-Ginzburg-Ziv Theorem, \emph{Combinatorica}  \textbf{26} (2006), no. 4, 445--453.


\bibitem{number-zs} D. J. Grynkiewicz, On the number of $m$-term zero-sum subsequences, \emph{Acta Arith.}  \textbf{121} (2006), no. 3, 275--298.

\bibitem{GrahamConj} D. J. Grynkiewicz, Note on a Conjecture of Graham, \emph{European J. Combin.}.
\textbf{32} (2011), no. 8,  1336--1344.

\bibitem{rasheed} D. J. Grynkiewicz and  R.  Sabar,  Monochromatic and zero-sum sets of nondecreasing modified diameter, \emph{Electron. J. Combin.} \textbf{13} (2006), no. 1, Research Paper 28 (electronic).

\bibitem{andy-paper} D. J. Grynkiewicz and A. Schultz, A Five Color Zero-Sum Generalization, in \emph{Graphs and Combin.} \textbf{22} (2006), 335--339.


\bibitem{oscar1} D. J. Grynkiewicz, O. Ordaz, M. T. Varela and F.  Villarroel, On Erd\H{o}s-Ginzburg-Ziv inverse theorems,  \emph{Acta Arith.}  \textbf{129} (2007), no. 4, 307--318.



\bibitem{oscar-weighted} D. J. Grynkiewicz, E. Marchan and O. Ordaz, Representation of finite abelian group elements by subsequence sums, \emph{J. Th\'eor. Nombres Bordeaux}  \textbf{21} (2009), no. 3, 559--587.



\bibitem{oscar-gaothms} D. J. Grynkiewicz, E. Marchan and O. Ordaz, A Weighted Generalization of Two Theorems of Gao, \emph{Ramanujan J.} \textbf{28} (2012), no. 3, 323-–340.



\bibitem{Gbook} D. J. Grynkiewicz, \emph{Structural Additive Theory}, Developments in Mathematics \textbf{30}, Springer (2013), Switzerland.

\bibitem{IttI} D. J. Grynkiewicz, Iterated Sumsets and Subsequence Sums, submitted.

\bibitem{ham-ordaz} Y. O. Hamidoune, O. Ordaz and A. Ortu\~no, On a combinatorial theorem of Erd\H{o}s, Ginzburg and Ziv, \emph{Combin. Probab. Comput.} \textbf{7} (1998), no. 4, 403-–412.

\bibitem{ham-subsumConj} Y. O. Hamidoune, Subsequence Sums, \emph{Combin.  Propab.  Comput.} \textbf{12} (2003), 413–-42.

\bibitem{ham-kst} Y. O. Hamidoune, Hyper-Atoms Applied to the Critical Pair Theory, to appear in \emph{Combinatorica}.



\bibitem{kisin} M.  Kisin, The  number  of  zero  sums  modulo
$m$ in  a  sequence  of  length $n$,  Mathematika  \textbf{41}
(1994), no. 1, 149-–163.

\bibitem{kneserstheorem} M. Kneser, Ein Satz \:uber abelsche Gruppen mit Anwendungen auf die Geometrie der Zahlen, \emph{Math. Z.} \textbf{61} (1955), 429--434.

\bibitem{Mann-preolson} H. B. Mann, Two addition theorems,
\emph{J. Combinatorial Theory}
 \textbf{3} (1967), 233–-235.


\bibitem{natboook} M. B. Nathanson, \emph{Additive Number Theory: Inverse Problems and the
Geometry of Sumsets}, Springer (1996), Harrisonburg, VA.




\bibitem{olson1-pregao} J. E. Olson,
An addition theorem for finite abelian groups,
\emph{J. Number Theory} \textbf{9} (1977), no. 1,  63-–70.


\bibitem{wolfgang-ordaz-olson-constant}
 O. Ordaz, A. Philipp, I. Santos, and W. A. Schmid,  On the Olson and the strong Davenport constants, \emph{J. Th\'eor. Nombres Bordeaux} 23 (2011), no. 3, 715-–750.

\bibitem{taobook} T. Tao and V. Vu, \emph{Additive Combinatorics}, Cambridge University Press (2006), Cambridge.

\bibitem{pingzhui-zeng}P. Yuan and  X. Zeng,    Two conjectures on an addition theorem, \emph{Acta Arith.} \textbf{148} (2011), no. 4, 395-–411.








\end{thebibliography}
\end{document}